\documentclass[11pt]{article}
\usepackage[a4paper,margin=2.6cm]{geometry}
\usepackage{amsmath,amssymb,amsthm}
\usepackage{booktabs}
\usepackage[expansion=false]{microtype}

\newtheorem{theorem}{Theorem}
\newtheorem{lemma}{Lemma}
\newtheorem{proposition}{Proposition}
\newtheorem{corollary}{Corollary}
\newtheorem{definition}{Definition}
\theoremstyle{remark}
\newtheorem{remark}{Remark}
\newtheorem{question}{Question}

\newcommand{\LB}{\mathrm{LB}}
\newcommand{\Vref}{V_{\mathrm{ref}}}
\newcommand{\diam}{\operatorname{diam}}
\newcommand{\eLB}{E_{\LB}}
\newcommand{\R}{\mathbb{R}}
\newcommand{\N}{\mathbb{N}}
\newcommand{\Cstab}{C_{\mathrm{stab}}}
\newcommand{\ub}{\bar u}

\title{The Certification Limits of KS-Type Layer Relaxations:\\
A Square-Root Ceiling and Its Breakdown}
\author{Guillaume Lecomte\thanks{Corresponding author. Email:
\texttt{guillaume.lecomteexed@edu.executive.em-lyon.com}.}}
\date{July 2026}

\begin{document}
\maketitle

\begin{abstract}
A separable bilinear layer relaxation assigns a population cost to each
layer and a nonpositive bilinear coupling to each pair of layers, then
excludes a profile span whenever the relaxed minimum exceeds a reference.
The central question addressed here is what a Kuznetsov--Sahinidis-type
layer relaxation can certify at best. We prove that this relaxation
architecture is subject to an intrinsic square-root ceiling. More
precisely, let $\gamma:=g(1)$ denote the singleton-layer cost. If the
relaxation is valid with $\gamma>-\tfrac12$, the reference is bounded per
unit size, and the underlying system admits two elementary families of
low-energy witness configurations, then
\[
N-\rho(N)\le 2\sqrt{vN}+O_v(1),
\qquad
v=\frac{\bar u+\max(\gamma,0)}{1+2\gamma},
\]
where $\bar u$ is the uniform reference bound. Hence no valid relaxation
in this class can certify a deficit growing faster than $O(\sqrt N)$,
independently of the detailed choice of the layer cost, interaction
kernel, or bounded reference sequence. No monotonicity, summability,
asymptotic growth, or shape assumption is imposed on the layer
functional. The threshold is sharp within the abstract witness-axiom class: at
$\gamma=-\tfrac12$, we construct a valid relaxation on a system
satisfying the witness axioms, with a bounded sound reference, for which
$N-\rho(N)=N-1$. For the geometric Lennard--Jones system, by contrast, a
three-atom chain witness restores the ceiling at the endpoint: the
counterexample separates the abstract class from the geometric one at
$\gamma=-\tfrac12$ exactly. Sufficiently far below the threshold the geometric protection fails as
well: we construct a valid relaxation of the Lennard--Jones system
itself, with the exact ground-state reference, whose deficit is linear
in $N$; and once the singleton cost falls below the negative optimal
stability constant, $\gamma\le-C$, no sound relaxation excludes
anything at all. The certifying power is thus reentrant in the singleton
cost, with one transition zone still open. Lennard--Jones
clusters provide the principal existing instance, while the witness
conditions also hold for broad normalized Morse and Mie families. Thus
every future valid KS-type construction for such systems inherits the
ceiling before its detailed constants are computed. The result is
one-sided: matching $\Theta(\sqrt N)$ laws require additional structure.
\end{abstract}

\noindent\textbf{MSC 2020:} 90C26 (primary); 82B05, 52C17 (secondary).\\
\textbf{Keywords:} Lennard--Jones clusters; layer relaxation; diameter
bounds; deterministic global optimization; relaxation limits.

\section{Introduction}\label{sec:intro}

A recurring pattern in certified computation is the one-dimensional layer
relaxation. An extended object of size $N$ is cut along one direction into
consecutive layers with populations $n_1,\dots,n_D\ge1$,
$\sum_kn_k=N$; each layer is charged a function $g(n_k)$ of its population
alone, each pair of layers a bilinear coupling $-\kappa(j-i)\,n_in_j$; and
a profile span $D$ is excluded whenever the minimum of the resulting
functional over all populations exceeds a reference value. The largest
non-excluded profile length, $\rho(N)$, is the output of the relaxation
procedure. In a geometric application, when the profile span controls a
genuine extent and the reference is sound, this output becomes a certified
upper bound on that extent.

The construction is not hypothetical. In the deterministic global
optimization of Lennard--Jones clusters, where certified boxes drive
branch-and-bound, the layer bound of Kuznetsov and Sahinidis \cite{ks2025}
is exactly of this form, and it is the principal certified diameter bound
developed for that problem; rigorous interatomic-distance bounds
\cite{xue,kiessling2025} complement it. The physically expected diameter
of a cluster grows only like $N^{1/3}$, so the accuracy of such
certificates is the natural question. Related work by the author
determines exact asymptotics for the concrete Kuznetsov--Sahinidis bound
\cite{lecomteB} and matching scaling laws under additional assumptions on
the kernel, the layer cost and the reference \cite{lecomteA}. The present
paper is logically independent of those analyses: none of their results
is used, and the kernel control and layer-cost estimates required here
are derived directly from validity against explicit witness
configurations.

This note removes the model-dependent assumptions on $g$ and $\kappa$ and
replaces them by two witness axioms on the underlying system. The
remaining hypotheses are validity, a lower bound on the singleton cost,
$\gamma:=g(1)>-\tfrac12$, boundedness of the reference per unit size, and
the existence of a low-energy adjacent pair and a low-energy chain
carrying an arm. No other property of the energy enters. The theorem
therefore knows nothing of the Lennard--Jones exponents and applies to any
geometric or nongeometric problem whose objects carry layer profiles and
admit the witnesses. Its conclusion is an explicit ceiling: no valid
relaxation with $\gamma>-\tfrac12$ can produce a deficit growing faster
than order $\sqrt N$, the constant depending on the reference bound and on
$\gamma$, and reducing to $2\sqrt{\ub\,N}$ in the normalized case
$\gamma=0$. The class need not be densely populated for this to be
informative. The Kuznetsov--Sahinidis bound \cite{ks2025} is the
principal implemented instance, and it is normalized; the theorem
identifies the obstruction before further examples are developed.

Two boundaries of the statement are drawn at the outset. The ceiling is
one-sided: it caps the deficit at $O(\sqrt N)$ for every relaxation above
the threshold, and weak relaxations may exclude nothing at all; matching
$\Theta(\sqrt N)$ laws require additional structural hypotheses and are
not consequences of the theorem. And the threshold is exact in the qualitative sense: the constant in our
bound diverges as $\gamma\downarrow-\tfrac12$, the endpoint
counterexample shows that some deterioration is unavoidable, and at
$\gamma=-\tfrac12$ the conclusion fails outright,
Section~\ref{sec:threshold} exhibiting a valid relaxation on a
witness-satisfying system, with a bounded sound reference, whose deficit
is $N-1$. The pair $\gamma>-\tfrac12$ and the counterexample at
$\gamma=-\tfrac12$ thus delimit the theorem completely within the
abstract class. The geometric Lennard--Jones system behaves differently
at the endpoint: a three-atom chain witness, unavailable abstractly,
forces enough coupling weight to restore the ceiling at $\gamma=-\tfrac12$
(Proposition~\ref{prop:geoendpoint}), so the counterexample measures a
genuine gap between the witness axioms and geometric validity. The
protection is not unconditional: sufficiently far below the threshold, a
valid relaxation of the geometric system itself, with the exact
ground-state reference, attains a linear deficit
(Proposition~\ref{prop:geofailure}). At the opposite extreme, when
$\gamma\le-C$, even the full singleton span is never excluded by any
sound reference, so the deficit vanishes identically
(Proposition~\ref{prop:deepcost}). The certifying power is therefore
reentrant in the singleton cost: zero, then potentially linear, then
universally at most of square-root order. The remaining transition zone
between the linear and square-root regimes is open. To the best of our knowledge, no ceiling of this generality appears
in the literature.

\section{The class}\label{sec:class}

\begin{definition}[Abstract layer system]\label{def:system}
An \emph{abstract layer system} is a class
$\mathcal X=\bigsqcup_{N\ge1}\mathcal X_N$ of objects called
configurations, an energy $E:\mathcal X\to\R$, and a profile map assigning
to each $X\in\mathcal X_N$ a tuple $n(X)=(n_1,\dots,n_D)$ of positive
integers with $\sum_kn_k=N$. The integer $D$ is the \emph{profile span}.
A geometric instance may additionally carry an extent functional $L(X)$
related to $D$; this extra structure is not used in the abstract theorem.
\end{definition}

\begin{definition}[Separable bilinear layer relaxation]\label{def:relax}
A \emph{separable bilinear layer relaxation} on a system is a pair
$(g,\kappa)$ with $g:\N\to\R$ and $\kappa:\{1,2,\dots\}\to[0,\infty)$,
acting through
\begin{equation}\label{eq:elb}
\eLB(n)\;=\;\sum_{k=1}^{D}g(n_k)\;-\;\sum_{1\le i<j\le D}\kappa(j-i)\,n_in_j.
\end{equation}
It is \emph{valid} if $E(X)\ge\eLB(n(X))$ for every $X\in\mathcal X$. Its
\emph{singleton cost} is $\gamma:=g(1)$, and it is \emph{normalized} if
$\gamma=0$. Given a reference sequence $\Vref(N)=-u_N\,N$ with $u_N\ge0$, a
profile span $D$ is \emph{excluded} at size $N$ when
\[
\min\{\eLB(n):n\in\N_{\ge1}^D,\ \sum_kn_k=N\}>\Vref(N).
\]
The \emph{relaxation span} is
$\rho(N)=\max\{D\le N:D\text{ not excluded}\}$, with $\rho(N):=0$ if
every span is excluded. Write $\ub:=\sup_Nu_N$. If the reference is sound
for the underlying optimization problem and the profile span controls a
geometric extent, then $\rho(N)$ has the corresponding certified
interpretation.
\end{definition}

Normalization is canonical for per-layer terms representing internal layer
energy, since a singleton layer carries no internal interaction; the
Kuznetsov--Sahinidis bound satisfies it identically \cite{ks2025}. It
is nevertheless a genuine restriction of the
abstract class and cannot be imposed by a harmless algebraic shift:
subtracting $\gamma$ from $g$ changes the functional by an amount
proportional to the profile span. The theorem below quantifies exactly how
much singleton cost the ceiling tolerates.

\begin{definition}[Witness axioms]\label{def:witness}
A system satisfies the witness axioms if:
\begin{itemize}
\item[(W1)] there is a configuration with profile $(1,1)$ and energy at
most $-1$;
\item[(W2)] for every $n\ge2$ there is a configuration with $m=2n$ layers,
profile all-ones except $n_c=n$ at $c=n$, and energy at most
$-(3n-2)$.
\end{itemize}
\end{definition}

The axioms concern the problem, not the relaxation: they assert that the
energy landscape contains a cheap adjacent pair and a cheap chain-with-arm
family. In the applications of Section~\ref{sec:instances} they are
verified by explicit constructions; abstractly they are the entire
contribution of the underlying energy.

\begin{remark}[Energy scale]\label{rem:scale}
The numerical constants in (W1) and (W2) fix units rather than the
mechanism. Analogous axioms with a well depth $\varepsilon>0$ and a witness
energy bounded above by a negative affine function of $n$ lead, after
rescaling, to the same square-root exponent with modified explicit
constants, the threshold on the singleton cost rescaling accordingly.
\end{remark}

\section{The ceiling}\label{sec:main}

Throughout this section $(g,\kappa)$ is a valid separable bilinear layer
relaxation on a system satisfying (W1) and (W2), $(u_N)$ is a bounded
reference sequence, and
\[
\gamma:=g(1),\qquad
\kappa_0:=1+2\gamma,\qquad
w:=\ub+\max(\gamma,0),\qquad
v:=\frac w{\kappa_0}.
\]

\begin{theorem}[Universal square-root ceiling above the threshold]
\label{thm:main}
If $\gamma>-\tfrac12$, then
\begin{equation}\label{eq:ceiling}
N-\rho(N)\;\le\;2\sqrt{v\,N}+O_v(1).
\end{equation}
In particular $N-\rho(N)=O(\sqrt N)$: no relaxation of the class with
singleton cost above the threshold produces a span below
$N-2\sqrt{v\,N}-O_v(1)$, and a deficit of order $N^{1-\delta}$ is
impossible for every $0<\delta<1/2$.
\end{theorem}

\begin{corollary}[Normalized case]\label{cor:normalized}
For $\gamma=0$, $v=\ub$ and
$N-\rho(N)\le2\sqrt{\ub\,N}+O_{\ub}(1)$.
\end{corollary}

\begin{proposition}[Explicit form]\label{prop:explicit}
In the setting of Theorem~\ref{thm:main},
$N-\rho(N)\le2\sqrt{v\,N}+35(v+1)$ for all $N\ge300\,(v+1)$, and
$N-\rho(N)\le300(v+1)$ trivially below that size. The constants are crude
and carry no significance beyond making the statement fully explicit; no
attempt has been made to optimize them.
\end{proposition}

\begin{remark}[Behaviour of the constant]\label{rem:constant}
The map $\gamma\mapsto v$ is decreasing on $\gamma\ge0$ when
$\ub\ge\tfrac12$, and in all cases $v\le\max(\ub,\tfrac12)$ there: a
positive singleton cost never weakens the ceiling beyond that value. On the
other side, $v\sim\ub/(1+2\gamma)$, so the leading coefficient
$2\sqrt v$ in our bound diverges like $2\sqrt{\ub/(1+2\gamma)}$ as
$\gamma\downarrow-\tfrac12$. The endpoint counterexample of
Section~\ref{sec:threshold} shows that some deterioration is unavoidable,
the conclusion failing outright at the threshold in the abstract class
(Proposition~\ref{prop:essential}); the failure does not extend to the
geometric Lennard--Jones system at the endpoint
(Proposition~\ref{prop:geoendpoint}), and whether the rate
$2\sqrt{\ub/(1+2\gamma)}$ is optimal above the threshold is not
addressed here.
\end{remark}

\begin{remark}[No soundness assumption]\label{rem:nosound}
The reference is arbitrary: nothing requires $\Vref(N)$ to bound the true
minimum of $E$, and no condition of the form
$\inf_{\mathcal X_N}E\le-\ub N$ is imposed. The ceiling is a statement
about the formal output of the exclusion procedure and therefore also
holds for unsound references. Soundness matters only when the relaxation
span is interpreted as a certificate. Likewise no stability of the energy
is assumed: boundedness of $(u_N)$ is an axiom here, and a theorem in the
Lennard--Jones instance (Remark~\ref{rem:stability}).
\end{remark}

The proof occupies the rest of this section. Two lemmas extract from
validity, tested against the witnesses, everything the argument needs; an
exclusion dichotomy replaces every summability hypothesis; and an explicit
two-pile profile produces the crossing.

\begin{lemma}[Linear intra-layer ceiling]\label{lem:intra}
Let $(g,\kappa)$ be valid on a system satisfying (W2). Then for every
$n\ge2$, with $m=2n$ and $c=n$,
\[
g(n)\;\le\;B_m+(n-1)\,\Phi_m(c)-(3n-2)-(2n-1)\gamma
\;\le\;B_m+(n-1)\,\overline\Phi_m-(3n-2)-(2n-1)\gamma,
\]
where $B_m=\sum_{d=1}^{m-1}(m-d)\kappa(d)$ and
$\Phi_m(c)=\sum_{j\ne c}\kappa(|c-j|)
\le\overline\Phi_m=2\sum_{d<m}\kappa(d)$.
\end{lemma}

\begin{proof}
Let $X$ be the witness of (W2). Writing $n_c=1+(n-1)$, the bilinear term
of \eqref{eq:elb} at its profile splits as
\[
\sum_{i<j}\kappa(j-i)\,n_in_j
=\sum_{i<j}\kappa(j-i)+(n-1)\sum_{j\ne c}\kappa(|c-j|)
=B_m+(n-1)\Phi_m(c),
\]
and the $m-1=2n-1$ singleton layers contribute $(2n-1)\gamma$, so
$\eLB(n(X))=(2n-1)\gamma+g(n)-B_m-(n-1)\Phi_m(c)$. Validity against
$E(X)\le-(3n-2)$ rearranges to the display.
\end{proof}

\begin{lemma}[Nearest-neighbour weight]\label{lem:unit}
Let $(g,\kappa)$ be valid on a system satisfying (W1). Then
$\kappa(1)\ge1+2\gamma=\kappa_0$; in particular $\kappa(1)>0$ whenever
$\gamma>-\tfrac12$, and $\kappa(1)\ge1$ in the normalized case.
\end{lemma}

\begin{proof}
At the witness of (W1), $\eLB=2\gamma-\kappa(1)$ and $E\le-1$; validity
rearranges to the claim.
\end{proof}

\paragraph{The exclusion dichotomy.}
Fix a size $N$. If the span $D=N$ is not excluded then $\rho(N)=N$ and the
deficit vanishes. Otherwise the exclusion inequality applies to the unique
profile at $D=N$, the all-ones profile, for which
$\eLB=N\gamma-B_N$; exclusion means $N\gamma-B_N>-u_N\,N$, that is
\begin{equation}\label{eq:K}
B_N\;<\;(u_N+\gamma)\,N .
\tag{K}
\end{equation}
Since $B_N\ge0$, exclusion of the full span forces $u_N+\gamma>0$; and
since each term of $B_N$ with $d\le N/2$ carries the weight $N-d\ge N/2$,
\begin{equation}\label{eq:mass}
\sum_{d\le N/2}\kappa(d)\;<\;2(u_N+\gamma)\;\le\;2w,
\end{equation}
the last step because $u_N+\gamma\le\ub+\gamma= w$ for $\gamma\ge0$ and
$u_N+\gamma<u_N\le\ub=w$ for $\gamma<0$. Summability of the kernel is
thus not assumed: bounded mass on the relevant range is what exclusion of
the full span costs. Both steps use only $\kappa\ge0$ and the definition
of the procedure; no witness and no energy.

\begin{lemma}[Linear ceiling under exclusion]\label{lem:linceiling}
Let $\gamma>-\tfrac12$, and suppose the span $D=N$ is excluded at size
$N$ on a system satisfying (W2). Then $g(n)\le8w\,n$ for every
$2\le n\le N/4$.
\end{lemma}

\begin{proof}
Let $m=2n\le N/2$. By \eqref{eq:mass},
$\overline\Phi_m=2\sum_{d<m}\kappa(d)\le4w$ and
$B_m\le m\sum_{d<m}\kappa(d)\le2wm$. Lemma~\ref{lem:intra} gives
\[
g(n)\le(n-1)\cdot4w+2w\cdot2n-(3n-2)-(2n-1)\gamma
=8wn-4w-3n+2-(2n-1)\gamma .
\]
For $\gamma\ge0$ the last term is nonpositive and $-3n+2<0$, so
$g(n)\le8wn$. For $-\tfrac12<\gamma<0$,
$-(2n-1)\gamma=(2n-1)|\gamma|<(2n-1)/2\le n$, so
$g(n)\le8wn-3n+2+n=8wn-2n+2\le8wn$ for $n\ge1$.
\end{proof}

\begin{proof}[Proof of Proposition~\ref{prop:explicit}, hence of
Theorem~\ref{thm:main} and Corollary~\ref{cor:normalized}]
Fix $N\ge300(v+1)$ and suppose $D=N$ is excluded, the other branch of the
dichotomy giving deficit zero. Set
\[
e:=\bigl\lceil2\sqrt{v\,N+1}\bigr\rceil+\lceil34v\rceil+5,
\qquad D:=N-e,
\]
and let $P$ be the profile on $D$ layers consisting of all ones except two
adjacent central layers of populations $1+\lfloor e/2\rfloor$ and
$1+\lceil e/2\rceil$. Since $v+1\le N/300$, we have
$e\le2\sqrt{v\,N}+34v+9\le\tfrac{2}{\sqrt{300}}N+\tfrac{34}{300}N+9
\le0.23\,N+9\le N/2-4$, so $D\ge2$, the profile is admissible, and both
pile populations lie in $[2,N/4]$, the range of
Lemma~\ref{lem:linceiling}.

The singleton layers of $P$ contribute
$(D-2)\gamma\le N\max(\gamma,0)$. Dropping the background and
pile-to-background terms of \eqref{eq:elb}, all nonpositive because
$\kappa\ge0$, and keeping the single pile-to-pile term at distance $1$,
\[
\eLB(P)\;\le\;N\max(\gamma,0)
+g\Bigl(1+\Bigl\lfloor\frac e2\Bigr\rfloor\Bigr)
+g\Bigl(1+\Bigl\lceil\frac e2\Bigr\rceil\Bigr)
-\kappa(1)\Bigl\lfloor\frac e2\Bigr\rfloor\Bigl\lceil\frac e2\Bigr\rceil,
\]
which, by Lemma~\ref{lem:linceiling}, Lemma~\ref{lem:unit}
($\kappa(1)\ge\kappa_0>0$) and
$\lfloor e/2\rfloor\lceil e/2\rceil\ge(e^2-1)/4$, is at most
\[
N\max(\gamma,0)+8w(e+2)-\kappa_0\,\frac{e^2-1}4
\;=\;N\max(\gamma,0)-\kappa_0\Bigl[\frac{e^2-1}4-8v(e+2)\Bigr],
\]
using $w=\kappa_0v$. Set $a:=2\sqrt{v\,N+1}$ and $b:=34v+5$, so that
$a+b\le e\le a+b+2$. Then
\[
\frac{e^2}4\;\ge\;\frac{(a+b)^2}4
\;=\;(v\,N+1)+\Bigl(17v+\tfrac52\Bigr)a+\frac{b^2}4,
\qquad
8v(e+2)\;\le\;8va+8vb+32v,
\]
and since $b/4-8v=v/2+5/4$,
\[
\frac{e^2-1}4-8v(e+2)
\;\ge\;v\,N+\tfrac34-32v+\Bigl(9v+\tfrac52\Bigr)a
+\Bigl(\frac v2+\frac54\Bigr)b
\;\ge\;v\,N+17v^2+13v+7\;\ge\;v\,N,
\]
using $a\ge0$ and $(\tfrac v2+\tfrac54)(34v+5)=17v^2+45v+\tfrac{25}4$.
Therefore
\[
\eLB(P)\;\le\;N\max(\gamma,0)-\kappa_0\,v\,N
\;=\;N\max(\gamma,0)-wN
\;=\;-\ub\,N\;\le\;-u_N\,N\;=\;\Vref(N),
\]
the span $D=N-e$ is not excluded, and
\[
N-\rho(N)\;\le\;e\;\le\;2\sqrt{v\,N+1}+34v+7
\;\le\;2\sqrt{v\,N}+35(v+1),
\]
by $\sqrt{x+1}\le\sqrt x+1$. For $N<300(v+1)$ the trivial bound
$N-\rho(N)\le N\le300(v+1)$ holds.
\end{proof}

\section{Instances}\label{sec:instances}

\subsection{Lennard--Jones clusters}\label{sec:lj}

The originating instance is geometric. Configurations are finite subsets
of $\R^3$; a configuration is oriented so that a diameter lies along the
$x$-axis and space is cut into unit layers $L_k=\{k-1\le x<k\}$; the
system consists of the fully spanning configurations, those occupying $D$
consecutive nonempty layers, for which the $x$-extent equals the diameter
and $D-2<\diam X<D$ \cite{ks2025}; the energy is
$E(X)=\sum_{i<j}V_{\mathrm{LJ}}(|x_i-x_j|)$,
$V_{\mathrm{LJ}}(r)=r^{-12}-2r^{-6}$, in units of the equilibrium
distance and well depth.

\begin{proposition}[Lennard--Jones witnesses]\label{prop:lj}
The oriented Lennard--Jones layer system satisfies (W1) and (W2).
Consequently Theorem~\ref{thm:main} applies formally to every valid
separable bilinear layer relaxation of the Lennard--Jones layer system
with singleton cost above the threshold, for every bounded reference
sequence. Its interpretation as a certified diameter bound additionally
requires a sound reference and the structural full-occupancy property of
minimizers established in \cite{ks2025}, together with the span-diameter
relation $D-2<\diam X<D$ recalled above.
\end{proposition}

\begin{proof}
\emph{(W1).} Let $X=\{(0,0,0),(1,0,0)\}$. The diameter is the segment
between the two atoms, of length $1$, along the $x$-axis; the atoms sit in
$L_1$ and $L_2$, so $X$ is fully spanning with profile $(1,1)$, and
$E(X)=V_{\mathrm{LJ}}(1)=-1$.

\emph{(W2).} Fix $n\ge2$, $m=2n$, $c=n$, and take
\[
a_j=\Bigl(j-\tfrac12,\,0,\,0\Bigr)\ (1\le j\le m),
\qquad
b_i=\Bigl(c-\tfrac12,\,i,\,0\Bigr)\ (1\le i\le n-1),
\]
a unit chain along the $x$-axis carrying a transverse arm at layer $c$.
Then $a_j\in L_j$ and $b_i\in L_c$, so the occupied layers are
$L_1,\dots,L_m$, consecutive and nonempty, with the profile of (W2). The
extremal chain pair has $|a_1-a_m|=m-1$; every other pair is shorter,
chain pairs trivially, arm pairs of length at most $n-2$, and mixed pairs
satisfying $|b_i-a_j|^2=(c-j)^2+i^2\le(m/2)^2+(n-1)^2\le(m-1)^2$, because
$n-1\le m/2-1$ gives $(m/2)^2+(m/2-1)^2=m^2/2-m+1\le m^2-2m+1$ for
$m\ge2$. Hence $\diam X=m-1$ along the $x$-axis and $X$ is fully spanning
with $D=m$. For the energy: chain pairs contribute
$\sum_{d=1}^{m-1}(m-d)V_{\mathrm{LJ}}(d)\le(m-1)V_{\mathrm{LJ}}(1)=-(m-1)$,
since $V_{\mathrm{LJ}}(d)<0$ for every integer $d\ge1$; the column of
layer $c$, namely $a_c$ with the arm, consists of $n$ collinear atoms at
unit spacing, contributing at most $-(n-1)$; and every remaining pair is
mixed with $j\ne c$, at distance $\sqrt{(c-j)^2+i^2}\ge\sqrt2>2^{-1/6}$,
contributing a negative amount. The three classes of pairs partition the
pair set, so no pair is counted twice, and $E(X)\le-(m-1)-(n-1)=-(3n-2)$.
\end{proof}

\begin{remark}[Stability and the constant]\label{rem:stability}
For a sound Lennard--Jones reference, $\Vref(N)\ge V^*_N\ge-\Cstab N$ by
stability of the potential \cite{ruelle}, so $\ub\le\Cstab$ and the
ceiling constant is absolute for normalized bounds.
\end{remark}

\begin{remark}[Sharpness of the two tests]\label{rem:sharp}
Lemma~\ref{lem:unit} is sharp: a normalized relaxation with
$\kappa(1)=1$, such as the conservative Lennard--Jones layer kernel
$\kappa(1)=1$, $\kappa(d)=-V_{\mathrm{LJ}}(d-1)$ for $d\ge2$, attains it
with equality, and the counterexample of Section~\ref{sec:threshold}
attains it with equality at the threshold. The chain-with-arm witness of
(W2) is the degenerate limit of the pile profiles that drive the actual
crossings of the implemented bounds \cite{ks2025}.
\end{remark}

\subsection{Other pair potentials}\label{sec:otherpair}

The verification above used exactly two properties of
$V_{\mathrm{LJ}}$: the unit well $V(1)=-1$, and $V\le0$ at every
non-bonded witness distance, all of which are integers $d\ge2$ or of the
form $\sqrt{(c-j)^2+i^2}\ge\sqrt2$.

\begin{proposition}[Pair-potential instances]\label{prop:pairs}
Let $V$ be a pair potential, unit-normalized, with $V(1)=-1$ and
$V(r)\le0$ for all $r\ge\sqrt2$. Then the oriented layer system of $V$
satisfies (W1) and (W2), and Theorem~\ref{thm:main} applies to every valid
separable bilinear layer relaxation of its layer system with singleton
cost above the threshold and bounded reference. This covers the normalized
Morse potential at every stiffness $a>0$, for which
$V(r)=e^{-2a(r-1)}-2e^{-a(r-1)}=e^{-a(r-1)}\bigl(e^{-a(r-1)}-2\bigr)<0$
for all $r\ge1$, and every normalized Mie potential with repulsive
exponent strictly larger than its attractive exponent, which is negative
beyond its minimum; regularized Buckingham-type potentials scaled to
$V(1)=-1$ and satisfying the tail-sign condition are covered likewise,
though the unregularized potential, ill-behaved at short range, is not a
clean instance and nothing here depends on it. The proposition asserts
the witness axioms, not the existence of relaxations: any future valid
KS-type relaxation for these potentials with singleton cost above the
threshold would inherit the ceiling.
\end{proposition}

\begin{proof}
(W1) is the pair at distance $1$. In (W2) the bonded pairs contribute $-1$
each as before, and every non-bonded pair sits at a distance where
$V\le0$ by hypothesis; the counting is unchanged.
\end{proof}

\begin{remark}[Prospective scope]\label{rem:prospective}
The theorem is prospective as well as retrospective: it constrains not
only the existing Lennard--Jones layer bounds but every future KS-type
construction for any pair potential satisfying the witness conditions.
The logical order matters. The witnesses concern the problem, and their
existence makes the theorem applicable to every valid relaxation of the
class; it does not by itself produce one. Constructing a KS-type
relaxation for Morse, Mie or Buckingham requires verifying that the
transposed per-layer estimate and kernel are a valid lower bound over all
layer configurations, a per-potential verification not carried out here.
Proposition~\ref{prop:pairs} claims exactly the first implication: once
such a relaxation is built and its validity established, the $\sqrt N$
ceiling applies to it before any of its constants are computed.
\end{remark}

\begin{remark}[Beyond pair potentials]\label{rem:beyond}
For many-body energies the axioms must be checked against the witness
geometry, and can fail. The Stillinger--Weber energy is the instructive
case: its pair part qualifies, but the three-body term penalizes the right
angles of the chain-with-arm witness, so (W2) does not follow from the
configuration of Proposition~\ref{prop:lj}, and adapted witnesses,
respecting the preferred angles inside a slab, would be needed. We record
this as open: the theorem applies exactly when the witnesses exist, and
asserting them for a specific many-body model is a model-specific
verification, not a corollary. Conversely, nothing in
Sections~\ref{sec:class} and~\ref{sec:main} is geometric: any problem
whose feasible objects carry layer profiles and admit the two cheap
witness families falls under Theorem~\ref{thm:main}.
\end{remark}

\section{What the theorem does not say}\label{sec:onesided}

The ceiling is one-sided by design. It caps the deficit of every
relaxation with singleton cost above the threshold at $O(\sqrt N)$; it
does not assert that any given relaxation achieves the cap, and weak
relaxations may exclude nothing at all, with $\rho(N)=N$ and deficit
zero, a case the dichotomy of Section~\ref{sec:main} handles trivially.
Achievement, the matching lower bound $N-\rho(N)\ge c\sqrt N$ together
with exact constants, requires structure; such matching laws are
established under the additional hypotheses recalled in the introduction
and are neither used nor reproved here. Above the threshold,
therefore, no improvement of $g$, of $\kappa$ or of the reference can
produce a relaxation deficit asymptotically larger than order $\sqrt N$;
obtaining a larger deficit therefore requires either reaching or crossing
the singleton-cost threshold, which Section~\ref{sec:threshold} shows
suffices in the abstract setting, or modifying at least one of the
structural ingredients of the separable bilinear layer paradigm.

\section{The threshold is sharp}\label{sec:threshold}

The singleton cost enters the proof at three places: the pair test
supplies $\kappa(1)\ge1+2\gamma$, which is the entire quadratic
mechanism; the full-span all-ones energy supplies the kernel-mass bound
\eqref{eq:K} with budget $u_N+\gamma$; and the singleton layers of the
witnesses and of the two-pile profile carry the cost $\gamma$ per layer.
Theorem~\ref{thm:main} shows the bookkeeping closes for every
$\gamma>-\tfrac12$. At the threshold it fails in the abstract class, and
not for want of a better proof; the geometric Lennard--Jones system, by
contrast, retains the ceiling there (Proposition~\ref{prop:geoendpoint}),
loses it far below (Proposition~\ref{prop:geofailure}), and certifies
nothing once the singleton cost passes the stability constant
(Proposition~\ref{prop:deepcost}).

\begin{proposition}[Sharpness at the threshold]\label{prop:essential}
There exist an abstract layer system satisfying (W1) and (W2), a valid
separable bilinear layer relaxation on it with $\gamma=-\tfrac12$, and a
bounded reference sequence, sound for the system, such that
\[
N-\rho(N)\;=\;N-1\qquad\text{for every }N\ge2 .
\]
\end{proposition}

\begin{proof}
Take $g(n)=-2n+\tfrac32$ and $\kappa\equiv0$, so $\gamma=-\tfrac12$ and
$\kappa(1)=0=1+2\gamma$: Lemma~\ref{lem:unit} is saturated. Let the
system consist, for each $N\ge1$, of a configuration $A_N$ with profile
$(N)$ and energy $E(A_N)=-2N+\tfrac32$; of one configuration with profile
$(1,1)$ and energy $-1$; and, for each $n\ge2$, of the (W2) witness of
size $3n-1$ on $2n$ layers with energy $-(3n-2)$. Validity holds with
equality on every configuration of the system:
$\eLB(A_N)=g(N)=E(A_N)$; $\eLB(1,1)=2\gamma=-1$; and
$\eLB=g(n)+(2n-1)\gamma=(-2n+\tfrac32)-\tfrac{2n-1}2=-(3n-2)$ on the
(W2) witnesses. The axioms (W1) and (W2) hold, with equality. Take
$\Vref(N):=E(A_N)=-2N+\tfrac32=-u_NN$, so $u_N=2-\tfrac3{2N}\in(0,2)$ is
bounded, and the reference is sound: $A_N$ realizes the minimum of $E$
over the size-$N$ configurations of the system, since the (W2) witness of
size $N=3n-1$ has energy $-(3n-2)=-N+1>-2N+\tfrac32$ for $N\ge2$, and the
pair has energy $-1>-\tfrac52=E(A_2)$. Now for every profile of size $N$
and span $D$,
\[
\eLB(n_1,\dots,n_D)=\sum_{k=1}^D\Bigl(-2n_k+\tfrac32\Bigr)
=-2N+\tfrac32D,
\]
which equals $\Vref(N)$ at $D=1$ and strictly exceeds it for every
$D\ge2$. Every span $D\ge2$ is excluded, $\rho(N)=1$, and the deficit is
$N-1$.
\end{proof}

\begin{proposition}[The geometry retains the ceiling at and below the
endpoint]\label{prop:geoendpoint}
Let $(g,\kappa)$ be a valid separable bilinear layer relaxation of the
oriented Lennard--Jones system of Section~\ref{sec:lj} with
\[
-\tfrac{8319}{12288}\;<\;\gamma\;\le\;-\tfrac12,
\]
and let $(u_N)$ be a bounded reference sequence. Then, with
$\kappa_\star(\gamma):=\tfrac13\bigl(3\gamma+\tfrac{8319}{4096}\bigr)>0$
and $v_\star:=\ub/\kappa_\star(\gamma)$,
\[
N-\rho(N)\;\le\;2\sqrt{v_\star\,N}+35(v_\star+1)
\qquad\text{for all }N\ge300\,(v_\star+1);
\]
in particular, at the endpoint $\gamma=-\tfrac12$ one has
$v_\star=\lambda\ub$ with $\lambda:=4096/725$, and the endpoint
failure of Proposition~\ref{prop:essential} does not occur
geometrically.
\end{proposition}

\begin{proof}
The three-atom chain
$X_3=\{(\tfrac12,0,0),(\tfrac32,0,0),(\tfrac52,0,0)\}$ is fully spanning
with profile $(1,1,1)$, its diameter being the axial segment of length
$2$, and
\[
E(X_3)=2V_{\mathrm{LJ}}(1)+V_{\mathrm{LJ}}(2)
=-2-\tfrac{127}{4096}=-\tfrac{8319}{4096}.
\]
Validity at $X_3$ reads $3\gamma-2\kappa(1)-\kappa(2)\le E(X_3)$, that
is,
\[
2\kappa(1)+\kappa(2)\;\ge\;3\gamma-E(X_3)
\;=\;3\gamma+\tfrac{8319}{4096}\;=\;3\kappa_\star(\gamma)\;>\;0,
\qquad\text{hence}\qquad
\kappa(r)\;\ge\;\kappa_\star(\gamma)
\]
for some $r\in\{1,2\}$, fixed by the relaxation; at $\gamma=-\tfrac12$,
$\kappa_\star=\tfrac{725}{4096}=\lambda^{-1}$. This replaces
Lemma~\ref{lem:unit}, which at and below the endpoint yields only
$\kappa(1)\ge1+2\gamma\le0$. The remaining ingredients survive on the
whole stated range: the dichotomy gives, under exclusion of the full
span, $B_N<(u_N+\gamma)N$ with $\gamma<0$ and the mass bound
\eqref{eq:mass} with $w=\ub$; and in the proof of
Lemma~\ref{lem:linceiling} the absorption
$-(2n-1)\gamma=(2n-1)|\gamma|\le3n-2$ holds for every $n\ge2$ since
$|\gamma|<\tfrac{8319}{12288}<\tfrac{3n-2}{2n-1}$, so
$g(n)\le8\ub\,n$ for $2\le n\le N/4$. Run the proof of Proposition~\ref{prop:explicit} with
the two piles placed at layer distance $r$ instead of $1$, every other
inter-layer term being dropped by sign, and with $v:=\ub/\kappa(r)$: the
background singleton term $(D-2)\gamma$ is now nonpositive and is
dropped as well, and the computation gives
$\eLB(P)\le8\ub(e+2)-\kappa(r)\tfrac{e^2-1}4\le-\ub N\le\Vref(N)$ for
$e=\lceil2\sqrt{vN+1}\rceil+\lceil34v\rceil+5$, verbatim. Since
$v=\ub/\kappa(r)\le v_\star$, the display follows.
\end{proof}

\begin{remark}[Separation, and below the endpoint]\label{rem:separation}
Propositions~\ref{prop:essential} and~\ref{prop:geoendpoint} together
locate the counterexample precisely: at $\gamma=-\tfrac12$ the witness
axioms admit a linear deficit while geometric validity does not. The
separating object is a single further witness, the three-atom chain,
whose energy lies strictly below the additive pair count; no analogue is
available in the abstract class, where (W1) and (W2) can be satisfied
with equality. The proposition already exploits the test on its full range
$\gamma>-\tfrac{8319}{12288}\approx-0.677$, where the constant
$v_\star$ diverges; longer all-ones chains, whose per-atom
Lennard--Jones energy approaches $-(2\zeta(6)-\zeta(12))\approx-1.0344$,
would push the mechanism further, and we do not pursue them.
Proposition~\ref{prop:geofailure} shows the mechanism must in any case
stop: sufficiently far below the endpoint the geometric ceiling fails.
\end{remark}

\begin{proposition}[Geometric failure far below the threshold]
\label{prop:geofailure}
Let $V^*_N$ denote the Lennard--Jones ground-state energy and
$\Vref(N):=V^*_N$ the exact reference, sound by construction and bounded
by stability \cite{ruelle}. There exist a valid separable bilinear layer
relaxation $(g,\kappa\equiv0)$ of the oriented Lennard--Jones system and
a constant $\theta\ge\tfrac1{125}$ such that
\[
N-\rho(N)\;\ge\;\frac{\theta}2N-1
\qquad\text{for all $N$ large.}
\]
Its singleton cost satisfies $\gamma\le-q$, where
$q:=1+2\zeta(6)-\zeta(12)=2.03444\ldots$, far below the threshold
$-\tfrac12$; more generally, the underlying family realizes every
singleton cost in the open interval $(-C,-q)$, with $C\ge2.11867$ the
stability constant of the proof; the linear-deficit coefficient
degenerates as $\gamma\uparrow-q$, and the onset of exclusion becomes
non-uniform as $\gamma\downarrow-C$.
\end{proposition}

\begin{proof}
Set $K:=3+2\bigl(2\zeta(6)-\zeta(12)\bigr)=1+2q$ and
$C:=\sup_N\bigl(-V^*_N/N\bigr)$, finite by stability \cite{ruelle}. By
the definition of $C$, every finite Lennard--Jones configuration $Y$
satisfies $E(Y)\ge-C|Y|$, and the certified six-atom minimizer gives
$C\ge-V^*_6/6=2.11867\ldots>q$ \cite{ks2025}; write $\beta:=C-q>0$. By
superadditivity of $-V^*$ (two minimizers placed far apart interact
attractively, so $V^*_{M+N}\le V^*_M+V^*_N$), Fekete's lemma gives
$u_N:=-V^*_N/N\to C$; set $\delta_N:=C-u_N\ge0$, so $\delta_N\to0$,
without monotonicity.

\emph{Step 1: elongated configurations.} Let $X$ be fully spanning with
span $D\ge(1-\theta)N$, let $X'\subset X$ keep one atom per layer, and
put $\mathcal E:=X\setminus X'$, $|\mathcal E|=N-D\le\theta N$. Split the
pairs into $X'\!\times\!X'$, $X'\!\times\!\mathcal E$ and
$\mathcal E\!\times\!\mathcal E$. Atoms in layers $i,j$ are at distance
greater than $|i-j|-1$; moreover $-V_{\mathrm{LJ}}\le1$ everywhere and
$-V_{\mathrm{LJ}}$ is decreasing on $[1,\infty)$. Hence the kept atoms
satisfy
\[
-E(X')\;\le\;\sum_{d=1}^{D-1}(D-d)\,b(d)\;\le\;qD,
\qquad
b(1):=1,\quad b(d):=-V_{\mathrm{LJ}}(d-1)\ (d\ge2),
\]
since $\sum_{d\ge1}b(d)=1+\sum_{m\ge1}\bigl(2m^{-6}-m^{-12}\bigr)=q$;
each atom of $\mathcal E$ interacts with $X'$ by at most
$3+2\sum_{m\ge1}\bigl(2m^{-6}-m^{-12}\bigr)=K$, since at most three
kept atoms lie at layer distance at most one and at most two kept atoms
lie at each layer distance $d\ge2$; and
$-E(\mathcal E)\le C\,|\mathcal E|$ by stability. Hence
\[
-E(X)\;\le\;qD+(C+K)\,|\mathcal E|
\;\le\;\bigl[q+(C+K-q)\,\theta\bigr]N .
\]

\emph{Step 2: the relaxation.} Take $\kappa\equiv0$ and
$g(n):=-cn+\beta$ with $c:=C+\beta(1-\theta)$ and
\[
\theta:=\frac{C-q}{2\bigl(K+2(C-q)\bigr)}
\;\ge\;\frac{0.0842}{2\,(5.069+0.169)}\;>\;\frac1{125},
\]
the lower bound because $\theta$ is increasing in $C-q\ge0.0842$. Then
$\eLB(n)=-cN+\beta D$ for every profile of size $N$ and span $D$.
Validity: for $D<(1-\theta)N$,
$-E(X)+\beta D\le CN+\beta(1-\theta)N=cN$, since $E(X)\ge-CN$; for
$D\ge(1-\theta)N$, Step~1 gives
$-E(X)+\beta D\le\bigl[q+(C+K-q)\theta+\beta\bigr]N\le cN$, the last
inequality being $\theta\bigl(K+2(C-q)\bigr)\le C-q$, true with room by
the choice of $\theta$ (using $\beta=C-q$, so that
$C+K-q+\beta=K+2(C-q)$).

\emph{Step 3: exclusion.} The span $D$ is excluded at size $N$ exactly
when $\beta D>(c-u_N)N=\bigl(\beta(1-\theta)+\delta_N\bigr)N$. For $N$
large enough that $\delta_N\le\beta\theta/2$, every $D>(1-\theta/2)N$ is
excluded, while $D=1$ is not; hence $\rho(N)\le(1-\theta/2)N+1$ and the
deficit is at least $\theta N/2-1$. Finally
$\gamma=g(1)=-c+\beta=-C+\beta\theta\le-C+\beta=-q$.

\emph{Step 4: the interval covered.} The construction is a family. For
$\eta\in(0,1)$ and $a>0$, replace $\beta$ by $a$ and $\theta$ by
$\theta_a:=(1-\eta)(C-q)/(C+K-q+a)$, which satisfies the validity
constraint of Step~2 with slack $\eta$; then $\gamma=-C+a\theta_a$
increases continuously from $-C$ (as $a\downarrow0$) to $-q-\eta(C-q)$
(as $a\uparrow\infty$). Letting $\eta\downarrow0$, the union of these
families realizes every singleton cost in the open interval $(-C,-q)$.
The linear-deficit coefficient tends to zero as $\gamma\uparrow-q$. As
$\gamma\downarrow-C$, the coefficient $a$ of the span term tends to
zero and the onset of the asymptotic exclusion, the condition
$\delta_N\le a\theta_a/2$, becomes non-uniform; at $\gamma=-C$ the
full singleton span is never excluded
(Proposition~\ref{prop:deepcost}).
\end{proof}

\begin{remark}[Sound but not algorithmic]\label{rem:nonalgorithmic}
The relaxation of Proposition~\ref{prop:geofailure} is mathematically
explicit but not implementable: its slope involves the bulk constant
$C=\lim(-V^*_N/N)$, known only through bounds, and the exact reference
$V^*_N$ is precisely the quantity that certified computation seeks. The
construction is an existence statement about the class, not a usable
bound; soundness and boundedness are nonetheless exact properties of it,
which is all the refutation of the universal geometric statement
requires. Note also that Theorem~\ref{thm:main} never applied to it: the
witnesses hold geometrically, but the singleton cost sits far below the
threshold.
\end{remark}

\begin{proposition}[Very deep singleton costs certify nothing]
\label{prop:deepcost}
For the oriented Lennard--Jones system, every sound reference
($\Vref(N)\ge V^*_N$) and every relaxation of the class with singleton
cost $\gamma\le-C$ satisfy $\rho(N)=N$ for all $N$: the deficit is
identically zero. Validity is not even needed.
\end{proposition}

\begin{proof}
At the all-ones profile,
$\eLB(1,\dots,1)=N\gamma-B_N\le N\gamma\le-CN\le V^*_N\le\Vref(N)$,
using $\kappa\ge0$ and $V^*_N\ge-CN$; the span $D=N$ is never excluded.
\end{proof}

The certifying power of the class over the geometric system is therefore
reentrant in the singleton cost: identically zero deficit for
$\gamma\le-C$, linear deficit realizable on $(-C,-q)$, undetermined on
$[-q,-\tfrac{8319}{12288}]$, and at most $O(\sqrt N)$ for
$\gamma>-\tfrac{8319}{12288}$. An excessively negative singleton
penalty disarms the relaxation entirely; an intermediate one permits a
linear deficit; and the geometry restores the square-root ceiling as
$\gamma$ rises. The precise object is the set $\mathcal F_{\mathrm{LJ}}$ of those
$\gamma\in\R$ for which some valid Lennard--Jones layer relaxation with
singleton cost $\gamma$ and sound bounded reference satisfies
$\limsup_{N\to\infty}(N-\rho(N))/N>0$.
Propositions~\ref{prop:geofailure} and~\ref{prop:deepcost} and the
ceilings of Theorem~\ref{thm:main} and
Proposition~\ref{prop:geoendpoint} give
\[
(-C,-q)\subset\mathcal F_{\mathrm{LJ}},
\qquad
\mathcal F_{\mathrm{LJ}}\cap\Bigl((-\infty,-C]\cup
\bigl(-\tfrac{8319}{12288},\infty\bigr)\Bigr)=\varnothing .
\]

\begin{question}[Geometric transition]\label{q:geo}
Determine $\mathcal F_{\mathrm{LJ}}\cap[-q,-\tfrac{8319}{12288}]$. In
particular, is $\mathcal F_{\mathrm{LJ}}$ a single interval of the form
$(-C,\gamma_c)$?
\end{question}

\section*{Declarations}

\paragraph{Data availability.} This article contains no associated data:
all results are mathematical proofs, and every numerical constant used
($V^*_6$, $\zeta(6)$, $\zeta(12)$, the rational thresholds) is either
quoted from the cited literature or computable in closed form from the
expressions given in the text.

\paragraph{Conflict of interest.} The author declares no competing
interests.

\paragraph{Funding.} This research received no external funding.

\end{document}